\documentclass[a4paper,oneside,11pt]{amsart}

\usepackage{ifpdf}

\ifpdf
\usepackage[colorlinks,hyperindex,linkcolor=blue,urlcolor=black,citecolor=blue,
pdftitle={Resolvent criteria for similarity to a normal operator with
  spectrum on a curve},
pdfauthor={Michael A. Dritschel, Daniel Estevez, Dmitry Yakubovich},
backref
]{hyperref}
\fi

\usepackage{amsmath}
\usepackage{amsfonts}
\usepackage{amssymb}
\usepackage{amsthm}

\usepackage[english]{babel}
\usepackage{enumerate}

\usepackage{graphicx}
\usepackage{color}

\usepackage[abbrev]{amsrefs}

\newcommand{\T}{\mathbb{T}}
\newcommand{\D}{\mathbb{D}}
\newcommand{\C}{\mathbb{C}}
\newcommand{\R}{\mathbb{R}}
\newcommand{\Z}{\mathbb{Z}}
\newcommand{\B}{\mathcal{B}}

\renewcommand{\Re}{\operatorname{Re}}
\renewcommand{\Im}{\operatorname{Im}}

\newcommand{\dist}{\operatorname{dist}}

\theoremstyle{plain}
\newtheorem{theorem}{Theorem}
\newtheorem{lemma}[theorem]{Lemma}
\newtheorem{corollary}[theorem]{Corollary}

\newtheorem{theoremvc}{Theorem}

\newtheorem{theoremvce}{Theorem}

\newtheorem{theoremn}{Theorem}

\newtheorem*{theoremA}{Theorem A}

\theoremstyle{remark}
\newtheorem{example}{Example}

\newtheorem*{remarks}{Remarks}

\theoremstyle{definition}
\newtheorem*{definition}{Definition}

\begin{document}

\title[Resolvent criteria for similarity to a normal
operator]{Resolvent criteria for similarity to a normal operator with
  spectrum on a curve}



\author{Michael A. Dritschel}
\address{
  School of Mathematics, Statistics and Physics\\
  Herschel Building\\
  University of Newcastle\\
  Newcastle upon Tyne\\
  NE1 7RU\\
  UK}
\email{michael.dritschel@ncl.ac.uk}

\author{Daniel Est\'evez}
\address{GMV Innovating Solutions, Tres Cantos 28760 (Madrid), Spain}
\email{daniel@destevez.net}

\author{Dmitry Yakubovich}
\address{Departamento de Matem\'{a}ticas\\
  Universidad Aut\'onoma de Madrid\\ Cantoblanco 28049 (Madrid)\\ Spain\\
  and Instituto de Ciencias Matem\'{a}ticas (CSIC - UAM - UC3M - UCM)}
\email{dmitry.yakubovich@uam.es}

\thanks{The second author was supported by a grant from the
  Mathematics Department of the Universidad Aut\'onoma de Madrid and
  the Project MTM2015-66157-C2-1-P of the Ministry of Economy and
  Competitiveness of Spain.  This work forms part of his thesis,
  defended in 2017.  The third author was supported by the Project
  MTM2015-66157-C2-1-P and by the ICMAT Severo Ochoa project
  SEV-2015-0554 of the Ministry of Economy and Competitiveness of
  Spain and the European Regional Development Fund (FEDER)}


\date\today

\begin{abstract}
  We give some new criteria for a Hilbert space operator with spectrum
  on a smooth curve to be similar to a normal operator, in terms of
  pointwise and integral estimates of the resolvent.  These results
  generalize criteria of Stampfli, Van~Casteren and Naboko, and
  answers several questions posed by Stampfli in~\cite{Stampfli3}.
  The main tools are from our recent results~\cite{article2} on
  dilation to the boundary of the spectrum, along with the Dynkin
  functional calculus for smooth functions, which is based on
  pseudoanalytic continuation.
\end{abstract}

\maketitle

\section{Introduction}

As Stampfli proved in 1969 (see~\cite{Stampfli}), if $\Gamma \subset
\C$ is a smooth curve, $T$ is a bounded operator on a Hilbert space
$H$ with spectrum $\sigma(T)$ contained in $\Gamma$, and there is a
neighborhood $U$ of $\Gamma$ such that $\|(T-\lambda)^{-1}\|\leq
\dist(\lambda,\Gamma)^{-1}$ for all $\lambda \in U\setminus \Gamma$,
then $T$ is normal.  Theorems of this type were first proved by
Nieminen~\cite{Nieminen} for the case $\Gamma = \R$ and by
Donoghue~\cite{Donoghue} for the case when $\Gamma$ is a circle.

If $\Gamma$ is not smooth, such a result need no longer be true.  A
counterexample can be found in~\cite{StampfliHyp}.  Even if $\Gamma$
is a circle, the condition $\|(T-\lambda)^{-1}\| \leq
C\dist(\lambda,\Gamma)^{-1}$, $\lambda \in\C\setminus\Gamma$, where
$C$ is a constant greater than $1$, is not sufficient for $T$ to be
\emph{similar} to a normal operator; that is, for some invertible $S$
and normal operator $N$, to have $T = SNS^{-1}$.  See the paper
Markus~\cite{Markus}.  Benamara and
Nikolski~\cite{BenamaraNikolski}*{Section 3.2} have a general result
in this direction, and in a related article~\cite{NikolskiTreil},
Nikolski and Treil give a counterexample where $T$ is a rank one
perturbation of a unitary operator with $\sigma(T) \subset \T$.

Nevertheless, the hypothesis in Stampfli's theorem can be successfully
weakened.  Denote by $\B(H)$ the set of bounded linear operators on a
Hilbert space $H$.  We will prove the following:

\begin{theorem}
  \label{two-sided}
  Let $\Gamma \subset \C$ be a $C^{1+\alpha}$ Jordan curve, and
  $\Omega$ the domain it bounds.  Let $T\in \B(H)$ be an operator with
  $\sigma(T)\subset \Gamma$.  Assume that
  \begin{equation*}
    \|(T-\lambda)^{-1}\|\leq \frac{1}{\dist(\lambda,\Gamma)},\qquad
    \lambda\in U\setminus\overline{\Omega},
  \end{equation*}
  for some open set $U$ containing $\partial\Omega$, and
  \begin{equation*}
    \|(T-\lambda)^{-1}\|\leq \frac{C}{\dist(\lambda,\Gamma)},\qquad
    \lambda\in\Omega,
  \end{equation*}
  for some constant $C$.  Then $T$ is similar to a normal operator.
\end{theorem}

In other words, we assume that a resolvent estimate with constant $1$
is satisfied outside $\Omega$ and an estimate with constant $C$ is
satisfied inside $\Omega$ (then $C\ge 1$).  The same conclusion holds
if, vice versa, these estimates hold with with constant $1$ inside
$\Omega$ and with constant $C$ outside $\Omega$; see the Remark at the
end of Section~\ref{proof-thm-1}.  The result gives a positive answer
to Question~2 posed by Stampfli in~\cite{Stampfli3}, which he observed
as being the case when $\Gamma$ is a circle.

The proof of Theorem~\ref{two-sided} will use a generalization of a
theorem of Putinar and Sandberg on complete $K$-spectral sets that was
proved in~\cite{article2}.  In fact, this theorem is an easy corollary
of this generalization and Lemma~\ref{lemma-two-sided}, which is
stated below.  The connection of spectral sets and similarity problems
was already observed by Stampfli in~\cite{Stampfli3}.  In Theorem~8 of
that paper he proved via different techniques a version of our
Lemma~\ref{lemma-two-sided} under the assumption that
$\overline{\Omega}$ set is a spectral set for $T$ rather than a
$K$-spectral set, along with stronger smoothness conditions for the
boundary of $\Omega$.

Many different kinds of conditions implying normality of an operator
have been studied.  See, for instance,~\cite{Berberian3} and the
previous articles in this series.  In~\cite{CampbellGellar2}, Campbell
and Gellar studied operators $T$ for which $T^*T$ and $T+T^*$ commute,
showing, for instance, that if $\sigma(T)$ is a subset of a vertical
line or $\R$, then $T$ is normal.  In~\cite{Djordjevic} Djordjevi\'{c}
gave several conditions for an operator to be normal using the
Moore-Penrose inverse.  Gheondea considered operators which are the
product of two normal operators in~\cite{Gheondea}.  See
also~\cite{Moslehian} and references therein.

Here we exhibit conditions for an operator to be similar to a normal
operator in terms of estimates of (or more properly, bounds on) its
resolvent.  Others have done likewise.  In~\cite{VanCasteren80},
Van~Casteren proved the following.

\begin{theoremvce}
  \label{vc80}
  Let $T\in\B(H)$ be an operator with $\sigma(T)\subset\T$.  Assume
  that $T$ satisfies the resolvent estimate
  \begin{equation*}
    \|(T-\lambda)^{-1}\|\leq C(1-|\lambda|)^{-1},\qquad |\lambda|<1
  \end{equation*}
  and
  \begin{equation*}
    \|T^n\| \leq C,\qquad n\geq 0.
  \end{equation*}
  Then $T$ is similar to a unitary operator.
\end{theoremvce}

An operator satisfying the last condition in this theorem is said to
be power bounded.  Van~Casteren improved this result
in~\cite{VanCasteren83}, as follows.

\begin{theoremvc}
  \label{vc83}
  Let $T\in\B(H)$ be an operator with $\sigma(T)\subset\T$.  Assume
  that $T$ satisfies the resolvent estimate
  \begin{equation*}
    \|(T-\lambda)^{-1}\|\leq C(1-|\lambda|)^{-1},\qquad |\lambda|<1,
  \end{equation*}
  and for $1 < r < 2$ and $x\in H$,
  \begin{equation*}
    \int_{|\lambda|=r} \|(T-\lambda)^{-1}x\|^2\,|d\lambda| \leq
    \frac{C\|x\|^2}{r-1}  \quad \text{and}\quad
    \int_{|\lambda|=r} \|(T^*-\lambda)^{-1}x\|^2\,|d\lambda| \leq
    \frac{C\|x\|^2}{r-1}.
  \end{equation*}
  Then $T$ is similar to a unitary operator.
\end{theoremvc}

By writing the power series for the resolvent, one can check that
every power bounded operator satisfies the last two conditions in this
theorem.

A related result was proved independently by Naboko in~\cite{Naboko}.

\begin{theoremn}
  \label{n}
  Let $T\in\B(H)$ be an operator with $\sigma(T)\subset\T$.  Assume
  that $T$ satisfies the resolvent conditions
  \begin{equation*}
    \int_{|\lambda|=r} \|(T-\lambda)^{-1}x\|^2\,|d\lambda| \leq
    \frac{C\|x\|^2}{r-1}, \qquad 1 < r < 2,\ x\in H,
  \end{equation*}
  and
  \begin{equation*}
    \int_{|\lambda|=r} \|(T^*-\lambda)^{-1}x\|^2\,|d\lambda| \leq
    \frac{C\|x\|^2}{1-r}, \qquad r < 1,\ x\in H.
  \end{equation*}
  Then $T$ is similar to a unitary operator.
\end{theoremn}

Each of Theorems~\ref{vc80}, \ref{vc83} and \ref{n} in fact gives
necessary and sufficient conditions for similarity to a unitary
operator.

In these theorems, it is possible to replace $T$ by $T^*$, $T^{-1}$ or
$T^{*-1}$, obtaining yet other criteria.  For related results and
conditions, we refer the reader to~\cite{Malamud2}, where some results
close to Naboko's were independently obtained, and
to~\cite{NikolskiBook}*{Section 1.5.6}.  The conditions in
Theorems~\ref{vc83} and \ref{n} are not comparable, in that there is
no easy way of deducing one from the other.

Additionally, we present analogues of criteria of Van~Casteren and
Naboko, generalized from the circle to a smooth curve $\Gamma$.  The
corresponding integral conditions use the existence of a family of
curves, tending ``nicely'' to $\Gamma$ (from both sides) in place of
circles $|\lambda| = r$; details are given at the beginning of
Section~\ref{meansquare}.

The paper is organized as follows.  Sections~\ref{Sect-Dynkin}
and~\ref{sect-trasplantation} are preparatory.  The first of these
contains the basic facts about the pseudoanalytic extension of
functions and Dynkin's functional calculus for an operator $T$ with
first order resolvent growth near the spectrum (that is, growth which
is linear in the resolvent).  In Section~\ref{sect-trasplantation}, we
use this calculus to show that the resolvent estimates for an operator
$T$ with $\sigma(T)\subset \Gamma$ are equivalent to corresponding
resolvent estimates for $\eta(T)$, where $\eta$ is a smooth
diffeomorphism from $\Gamma$ to $\T$, so that $\sigma(\eta(T))\subset
\T$.  Section~\ref{proof-thm-1} deals with the proof of
Theorem~\ref{two-sided}, while in Section~\ref{meansquare}, we
formulate and prove analogues of mean-square criteria by Van~Casteren
and Naboko.  Finally, Section~\ref{sec-examples} contains a brief
discussion of related results in the literature and a few examples.

The authors are grateful to Maria Gamal' for several insightful
remarks and pointers to the literature.

\section{Dynkin's functional calculus}
\label{Sect-Dynkin}

Our key technical tool will be a generalization of the Riesz-Dunford
functional calculus as defined by Dynkin in~\cite{Dynkin72} using the
Cauchy-Green formula.  Before going into details, we need to set down
some definitions and notation.

Let $\Gamma \subset \C$ be a Jordan curve of class $C^{1+\alpha}$,
$0<\alpha<1$.  This means that $\Gamma$ is the image of $\T$ under a
bijective map $\psi:\T\to\Gamma$ such that $\psi \in C^1(\T)$, $\psi'$
does not vanish and $\psi'$ is H\"older $\alpha$; that is,
\begin{equation*}
  |\psi'(z)-\psi'(w)|\leq C|z-w|^\alpha, \qquad z,w\in\T.
\end{equation*}

A function $f:\Gamma \to \C$ is said to belong to
$C^{1+\alpha}(\Gamma)$ if $f\circ\psi \in C^1(\T)$ and $(f\circ
\psi)'$ is H\"older $\alpha$.  As an important example of such a
function, take $f = \psi^{-1}$.  This function has the additional
properties that $f(\Gamma) = \mathbb T$, $f^{-1}$ exists as a map from
$\mathbb T$ to $\Gamma$ and is differentiable.

The norm for $f \in C^{1+\alpha}(\Gamma)$ is defined as
\begin{equation*}
  \|f\|_{C^{1+\alpha}(\Gamma)} = \|f \circ \psi\|_{C(\T)} +
  \|(f\circ\psi)'\|_{C(\T)} + \|(f \circ \psi)'\|_\alpha,
\end{equation*}
where
\begin{equation*}
  \|g\|_\alpha = \sup_{z,w \in \T, z\neq w}
  \frac{|g(z)-g(w)|}{|z-w|^\alpha}.
\end{equation*}
The definition of this norm depends on the choice of the
parametrization $\psi$, but different choices yield equivalent norms.

Let $T \in \B(H)$ be an operator with $\sigma(T) \subset \Gamma$,
where $\Gamma$ is a Jordan curve of class $C^{1+\alpha}$.  Assume that
$T$ satisfies the following resolvent growth condition:
\begin{equation}
  \label{eq:resolvent-estimate}
  \|(T-\lambda)^{-1}\|\leq \frac{C}{\dist(\lambda,\Gamma)},\qquad
  \lambda \in \C\setminus\Gamma.
\end{equation}
Following Dynkin~\cite{Dynkin72}, a $C^{1+\alpha}(\Gamma)$ functional
calculus for $T$ can be defined.  Dynkin defines his calculus for a
scale of function algebras including $C^{1+\alpha}$ and operators
satisfying other resolvent estimates \eqref{eq:resolvent-estimate}.
Only the case relevant to this paper is discussed here.

To begin, recall the notion of pseudoanalytic extension.  If $f \in
C^{1+\alpha}(\Gamma)$, then by~\cite{Dynkin74}*{Theorem~2} there is a
function $F \in C^1(\C)$ such that $F|\Gamma = f$ and
\begin{equation}
  \label{eq:pseudoanalytic-estimate}
  \left|\frac{\partial F}{\partial \overline{z}}(z)\right|
  \leq C\|f\|_{C^{1+\alpha}(\Gamma)} \dist(z,\Gamma)^\alpha.
\end{equation}
Here, $\frac{\partial}{\partial \overline{z}} =
\frac{1}{2}\left(\frac{\partial}{\partial x} + i
  \frac{\partial}{\partial y}\right)$ and $C$ is a constant depending
only on $\Gamma$.  Every such function $F$ which extends $f$ and
satisfies \eqref{eq:pseudoanalytic-estimate} is called \textit{a
  pseudoanalytic extension of} $f$.

Dynkin uses the pseudoanalytic extension $F$ to define the operator
$f(T)$ by means of the Cauchy-Green integral formula.  Let $D$ be a
domain with smooth boundary such that $\Gamma \subset D$, and define
\begin{equation*}
  f(T) = \frac{1}{2\pi i} \int_{\partial D} F(\lambda) (\lambda -
  T)^{-1}\,d\lambda - \frac{1}{\pi} \iint_D \frac{\partial F}{\partial
    \overline{z}}(\lambda) (\lambda - T)^{-1}\,dA(\lambda).
\end{equation*}
The inequality \eqref{eq:pseudoanalytic-estimate} for $F$ and the
resolvent estimate \eqref{eq:resolvent-estimate} for $T$ imply that
the second integral is well defined.  It is possible to prove that the
definition does not depend on the particular choice of $D$ or
pseudoanalytic extension $F$.

This calculus has the usual properties of a functional calculus: it is
continuous from $C^{1+\alpha}(\Gamma)$ to $\B(H)$, is linear and
multiplicative, and coincides with the natural definition of $f(T)$ if
$f$ is rational.  It also satisfies the spectral mapping property:
$\sigma(f(T)) = f(\sigma(T))$.

\section{Passing from $\Gamma$ to $\T$}
\label{sect-trasplantation}

We now explain how to use Dynkin's functional calculus to pass from an
operator $T$ with $\sigma(T) \subset \Gamma$ to an operator $A$ with
$\sigma(A)\subset \T$.  The main result of this section is
Theorem~\ref{thm-main}, which relates the estimates for the resolvents
of $T$ and $A$.  In this way, resolvent growth conditions for $T$
imply equivalent conditions for $A$, and conversely.  This equivalence
plays a key role in what follows.

The next lemma gives regularity conditions for a certain function
$\eta: \Gamma \to \T$, which enables the construction of the operator
$A = \eta(T)$ using the Dynkin functional calculus.

\begin{lemma}
  \label{lemma-eta}
  Let $\Gamma$ be a Jordan curve of class $C^{1+\alpha}$ and $\eta \in
  C^{1+\alpha}(\Gamma)$ a function such that $\eta(\Gamma) = \T$ and
  $\eta^{-1}:\T \to \Gamma$ exists and is differentiable.  Fix any
  pseudoanalytic extension of $\eta$ to $\C$, also denoted by $\eta$.
  Then there is a neighborhood $U$ of $\Gamma$ such that $\eta : U \to
  \eta(U)$ is a $C^1$ diffeomorphism, $\eta(U)$ is a neighborhood of
  $\T$, and $\eta$ is bi-Lipschitz in $U$; that is, there are positive
  constants $c$ and $C$ such that
  \begin{equation*}
    c|z-w| \leq |\eta(z)-\eta(w)| \leq C|z-w|, \qquad z,w\in U.
  \end{equation*}
\end{lemma}

A consequence of this lemma used frequently below is that for
$\lambda\in U$, $\dist(\lambda,\Gamma)$ and $\dist(\eta(\lambda),\T)$
are comparable.

\begin{proof}[Proof of Lemma~\ref{lemma-eta}]
  Since $\partial \eta/\partial \overline{z} \equiv 0$ on $\Gamma$,
  the condition that $\eta^{-1}:\T \to \Gamma$ is differentiable
  implies that the differential of $\eta$ is non-singular on $\Gamma$.
  Therefore, for each point $x \in \Gamma$, there is an open ball
  $B(x,r(x))$ of center $x$ and radius $r(x)$ such that $\eta$ is
  bi-Lipschitz on $B(x,r(x))$.  By a compactness argument, $\eta$ is
  Lipschitz on some neighborhood of $\Gamma$.

  Pass to a finite collection $\{x_j\}$ on $\Gamma$ such that the
  balls $B(x_j, r(x_j)/2)$ cover $\Gamma$ and put $\varepsilon_0 =
  \min r(x_j)/2$.  Since $\eta|\Gamma$ is injective,
  \begin{equation*}
    \delta := \min_{\substack{|x-y|\geq \varepsilon_0
        \\ x,y\in\Gamma}}|\eta(x)-\eta(y)| > 0.
  \end{equation*}
  It follows that there is some $\rho>0$ such that
  \begin{equation*}
    {\widetilde \delta} :=
    \min_{\substack{|x-y|\geq \varepsilon_0 \\
        {\dist(x,\Gamma) \leq \rho,\
          \dist(y,\Gamma) \leq \rho
        }
      } }
    |\eta(x)-\eta(y)| > 0.
  \end{equation*}
  Now check that $\eta$ is bi-Lipschitz on the open set
  \begin{equation*}
    W = \bigg(\bigcup_j B\Big(x_j, \frac{r(x_j)}{2}\Big)\bigg) \cap
    \{x \in \C :
    \dist(x,\Gamma) < \rho
    \}.
  \end{equation*}
  Given points $x,y \in W$, then either $|x-y|< \varepsilon_0$, so
  that $x,y$ both belong to the same ball $B(x_k,r(x_k))$, where
  $\eta$ is bi-Lipschitz, or $|x-y|\geq \varepsilon_0$.  In the latter
  case,
  \begin{equation*}
    |\eta(x)-\eta(y)|\geq
    {\boldmath \widetilde{\delta} \geq
      \widetilde{\delta}\big(\operatorname{diam} W\big)^{-1}\,|x-y|.
    }
  \end{equation*}

  The injectivity of $\eta$ follows from the bi-Lipschitz property.
  The fact that it is possible to choose $U \subset W$ so that $\eta$
  is a $C^1$ diffeomorphism of $U$ is true because the differential of
  $\eta$ is non-singular in some neighborhood of $\Gamma$.  Finally,
  since $\eta(\Gamma) = \T$ and $\eta$ is an open mapping by being
  bi-Lipschitz, $\eta(U)$ is an open neighborhood of $\T$.
\end{proof}

The next theorem relates the resolvents estimates for $T$ and
$\eta(T)$.

\begin{theorem}
  \label{thm-main}
  Let $\Gamma$, $\eta$ and $U$ be as in Lemma~\ref{lemma-eta} and $T
  \in \B(H)$ be an operator satisfying the resolvent estimate
  \eqref{eq:resolvent-estimate}.  Let the operator $\eta(T)$ be
  defined by the $C^{1+\alpha}$-functional calculus for $T$.  Then
  $\sigma(\eta(T)) \subset \T$, and for some $C\geq 1$ and depending
  on $\Gamma$, $\eta$ ,$T$, but not on $\lambda$ or $x$,
  \begin{equation*}
    C^{-1}\|(T-\lambda)^{-1}x\| \leq \|(\eta(T)-\eta(\lambda))^{-1}x\|
    \leq C\|(T-\lambda)^{-1}x\|,
    \quad \lambda \in U\setminus \Gamma,\ x\in H.
  \end{equation*}
\end{theorem}

\begin{proof}
  The fact that $\sigma(\eta(T)) \subset \T$ follows from the spectral
  mapping theorem for the Dynkin functional calculus.

  For $\lambda \in U\setminus\Gamma$, define functions
  $\varphi_\lambda, \psi_\lambda \in C^{1+\alpha}(\Gamma)$ by
  \begin{equation}
    \label{eq:phi-psi}
    \varphi_\lambda(z) = \frac{\eta(z)-\eta(\lambda)}{z-\lambda},
    \qquad
    \psi_\lambda(z) = \frac{z-\lambda}{\eta(z)-\eta(\lambda)}.
  \end{equation}
  The operators $\varphi_\lambda(T)$ and $\psi_\lambda(T)$ are thus
  defined.  In fact,
  \begin{equation*}
    \varphi_\lambda(T) = (\eta(T)-\eta(\lambda))(T-\lambda)^{-1},
    \qquad \psi_\lambda(T) = (T-\lambda)(\eta(T)-\eta(\lambda))^{-1}.
  \end{equation*}
  Hence it suffices to show that
  \begin{equation*}
    \|\varphi_\lambda(T)\|\leq C_0,\qquad \|\psi_\lambda(T)\| \leq C_0,
  \end{equation*}
  for $C_0$ independent of $\lambda$.

  In fact, the functions $\varphi_\lambda$ and $\psi_\lambda$ are in
  $U\setminus\{\lambda\}$ and since $\eta$ is bi-Lipschitz,
  \begin{equation*}
    |\varphi_\lambda(z)| \leq C_1, \qquad |\psi_\lambda(z)| \leq C_1,
    \qquad z \in U\setminus\{\lambda\}.
  \end{equation*}
  Let $D$ be a domain with smooth boundary such that $\Gamma \subset D
  \subset \overline{D} \subset U$ and $\varepsilon > 0$, to be chosen
  later.

  By the Dynkin functional calculus, for $\lambda \in D$ and for
  $\varepsilon$ chosen small enough so that
  $\overline{B(\lambda,\varepsilon)} \subset D$,
  \begin{equation*}
    \begin{split}
      \varphi_\lambda(T) = &\frac{1}{2\pi i}\int_{\partial D}
      \varphi_\lambda(z) (z-T)^{-1}\, dz -\frac{1}{2\pi
        i}\int_{\partial B(\lambda,\varepsilon)} \varphi_\lambda(z)
      (z-T)^{-1}\, dz\\
      &- \frac{1}{\pi} \iint_{D\setminus B(\lambda,\varepsilon)}
      \frac{\partial\varphi_\lambda}{\partial\overline{z}}(z)
      (z-T)^{-1}\,dA(z).
    \end{split}
  \end{equation*}
  The case $\lambda \notin D$ is similar.

  The norm $\|\varphi_\lambda(T)\|$ is bounded by estimating the three
  terms separately.  For the second term, if $\varepsilon <
  \dist(\lambda,\Gamma)$, then
  \begin{equation*}
    \int_{\partial B(\lambda,\varepsilon)}
    |\varphi_\lambda(z)|\|(z-T)^{-1}\|\,|dz|
    \leq C_2\varepsilon(\dist(\lambda,\Gamma)-\varepsilon)^{-1}.
  \end{equation*}
  Letting $\varepsilon \to 0$, it is seen that this term is
  negligible.

  The norm of the first term is bounded by
  \begin{equation*}
    \frac{1}{2\pi} \int_{\partial D} |\varphi_\lambda(z)|
    \|(z-T)^{-1}\|\,|dz| \leq C_3 \dist(\partial D, \Gamma)^{-1}
    \operatorname{length}(\partial D).
  \end{equation*}

  Finally, by using Lemma~\ref{lemma-eta}, the norm of the third term
  is bounded by
  \begin{equation*}
    \begin{split}
      & \iint_{D\setminus B(\lambda,\varepsilon)}
      \left|\frac{\partial\varphi_\lambda}{\partial\overline{z}}(z)
      \right| \|(z-T)^{-1}\|\,dA(z) \\
      \leq & \iint_D \frac{1}{|z-\lambda|}
      \left|\frac{\partial
          \eta}{\partial\overline{z}}(z)\right|\|(T-\lambda)^{-1}\|
      \,dA(z)  \\
      \leq & \,C_4\iint_D
      |z-\lambda|^{-1}\dist(z,\Gamma)^{\alpha-1}\,dA(z) \\
      \leq & \,C_5
      \iint_D |\eta(z)-\eta(\lambda)|^{-1}\dist(\eta(z),\T)^{\alpha-1}
      \,dA(z) \\
      \leq & \,C_6 \iint_{\eta(D)} |\zeta - \eta(\lambda)|^{-1}
      \dist(\zeta,\T)^{\alpha-1}\,dA(\zeta) \\
      \leq & \,C_6 \int_{a\leq
        |\zeta| \leq b} |\zeta-\eta(\lambda)|^{-1}
      |1-|\zeta||^{\alpha-1}\,dA(\zeta).
    \end{split}
  \end{equation*}
  The change of variables $\zeta = \eta(z)$ has been performed and
  choice $a < b$ is made so that the set $\eta(D)$ is contained in the
  annulus $a \leq |\zeta| \leq b$.  By Lemma~\ref{lemma-integral}
  below, the last term in this chain of inequalities is smaller than a
  constant which is independent of $\lambda$.  Thus
  $\|\varphi_\lambda(T)\|\leq C_0$, with $C_0$ independent of
  $\lambda$.

  The proof that $\|\psi_\lambda(T)\|\leq C_0$ is very similar, in
  this case using that
  \begin{equation*}
    \begin{split}
      \left|\frac{\partial \psi_\lambda}{\partial \overline
          z}(z)\right| &=
      \frac{|z-\lambda|}{|\eta(z)-\eta(\lambda)|^2}
      \left|\frac{\partial \eta}{\partial \overline z}(z)\right| \leq
      C_7|\eta(z)-\eta(\lambda)|^{-1} \left|\frac{\partial
          \eta}{\partial \overline z}(z)\right|.
    \end{split}
  \end{equation*}
  The remaining bounds are obtained in the same way.  The proof is
  finished (modulo the next lemma).
\end{proof}

\begin{lemma}
  \label{lemma-integral}
  Let $0 < a < 1 < b$ and $a \leq |w| \leq b$ and $-1<\beta <0$.  Then
  for some $C$ independent of $w$,
  \begin{equation*}
    \iint_{a \leq |z| \leq b} |z-w|^{-1}|1-|z||^\beta\,dA(z) \leq C.
  \end{equation*}
\end{lemma}

\begin{proof}
  Performing a rotation if necessary, take $w$ to be real and
  positive, so that $a \leq w \leq b$.  By passing to polar
  coordinates and using the inequality
  \begin{equation*}
    |r e^{i\theta} - w|^{-1} \leq C_0|r+i\theta - w|^{-1},
  \end{equation*}
  which is valid for $a \leq r \leq b$, the integral in the statement
  of the lemma is less than a constant times
  \begin{equation*}
    \iint_{[a,b]\times[-\pi,\pi]}|\zeta-w|^{-1}|1-\Re \zeta|^\beta
    \,dA(\zeta).
  \end{equation*}

  Now assume that $a \leq w \leq 1$ (the case $1 \leq w \leq b$ will
  be similar).  Estimate the integral by dividing the region of
  integration into two pieces.  Put $t = (w+1)/2$.  Then
  \begin{equation*}
    \begin{split}
      &\iint_{[a,t]\times[-\pi,\pi]} \frac{|1-\Re\zeta|^\beta}{|\zeta-w|}
      \,dA(\zeta) + \iint_{[t,b]\times[-\pi,\pi]}
      \frac{|1-\Re\zeta|^\beta}{|\zeta-w|} \,dA(\zeta) \\ 
      \leq &
      \iint_{[a,t]\times[-\pi,\pi]} \frac{|w-\Re \zeta|^\beta}{|\zeta-w|}
      \,dA(\zeta) + \iint_{[t,b]\times[-\pi,\pi]} \frac{|1-\Re
      \zeta|^\beta}{|\zeta-w|} \,dA(\zeta) \\ 
      \leq &
      \iint_{[a-1+w,t-1+w]\times[-\pi,\pi]}
      \frac{|1-\Re\zeta'|^\beta}{|\zeta'-w|} \,dA(\zeta') +
      \iint_{[t,b]\times[-\pi,\pi]}
      \frac{|1-\Re\zeta|^\beta}{|\zeta-w|} \,dA(\zeta) \\ 
      \leq & 2\iint_{
        [2a-1,b]\times[-\pi,\pi]}
      \frac{|1-\Re\zeta|^\beta}{|\zeta-w|} \,dA(\zeta),
    \end{split}
  \end{equation*}
  where we have performed the change of variables $\zeta' = \zeta - 1
  + w$ and used that $2a-1 \leq a-1+w$ and $t-1+w\leq b$.  By a change
  to polar coordinates $\zeta = 1+re^{i\theta}$, the single pole at
  $1$ is seen to be of order strictly between $-2$ and $0$, and so the
  last integral is finite.
\end{proof}

\section{The proof of Theorem~\ref{two-sided}}
\label{proof-thm-1}

We first recall the result from~\cite{article2} to be used in the
proof of the theorem.  The statement given here is for a
$C^{1+\alpha}$ domain, although the original was proved under weaker
regularity conditions (see~\cite{article2}*{Theorem~2}).

\begin{theoremA}
  \label{putinar-sandberg}
  Let $T \in \B(H)$ and $\Omega$ a Jordan domain of class
  $C^{1+\alpha}$.  Assume there is some $R > 0$ such that for every
  $\lambda \in \partial \Omega$ there is some point $\mu_k(\lambda)
  \in \C\setminus \overline{\Omega}$ such that
  $\dist(\mu_k(\lambda),\partial\Omega) = |\mu_k(\lambda)-\lambda| =
  R$ and $\|(T-\mu_k(\lambda))^{-1}\| \leq R^{-1}$.  Then
  $\overline{\Omega}$ is a complete $K$-spectral set for some $K>0$.
\end{theoremA}

In other words, the conclusion is that there exists a constant $K\geq
1$ such that
\begin{equation*}
  \|f(T)\| \leq K\|f\|_{H^\infty(\Omega)},
\end{equation*}
for every (matrix-valued) rational function $f$ with poles off of
$\overline{\Omega}$ (and hence for every $f$ which is continuous in
$\overline{\Omega}$ and analytic in $\Omega$).  This result
affirmatively answers Question~3, posed by Stampfli in
\cite{Stampfli3}.

A nice overview of complete $K$-spectral sets can be found in
\cite{Paulsen}*{Chapter~9}.  A result of this property for an operator
$T$ is that $T$ dilates to an operator similar to a normal operator
with spectrum in the boundary of the domain.  The additional
assumptions in Theorem~\ref{two-sided} will allow us to conclude that
the operator $T$ itself is similar to a normal operator.  Curiously,
this will require only knowing the weaker property that $T$ has
$\overline{\Omega}$ as a $K$-spectral set; in other words, that
$\|f(T)\| \leq K\|f\|_{H^\infty(\Omega)}$ only for scalar valued
rational functions with poles off of $\overline{\Omega}$.

\begin{lemma}
  \label{similar-unitary}
  Under the hypotheses of Theorem~\ref{thm-main}, if $\eta(T)$ is
  similar to a unitary operator, then $T$ is is similar to a normal
  operator.
\end{lemma}

\begin{proof}
  Replacing $T$ by $STS^{-1}$, where $S$ is such that $S\eta(T)S^{-1}
  = \eta(STS^{-1})$ is unitary, it can be assumed that $\eta(T)$ is
  unitary.  Then $(\eta|\Gamma)^{-1} \in C^{1+\alpha}(\T)$.  Choose
  some $\beta\in(0,\alpha)$.  Then $(\eta|\Gamma)^{-1}$ is in the
  class $C_0^{1+\beta}(\T)$, which consists of functions $g\in
  C^{1+\beta}(\T)$ such that $(g'(z)-g'(w))/|z-w|^\beta\to 0$ as
  $z,w\in \T$, $|z-w|\to 0$.  Hence one can choose a sequence of
  rational functions $\{r_n\}_{n=1}^\infty$ with poles off $\T$ such
  that $r_n$ tend to $(\eta|\Gamma)^{-1}$ in $C^{1+\beta}(\T)$ (this
  follows, for instance, from \cite{Katznel-book}*{Theorem 2.12}).
  Thus $r_n \circ \eta$ tend to the identity function in
  $C^{1+\beta}(\Gamma)$.  By continuity of the
  $C^{1+\beta}(\Gamma)$-functional calculus for $T$, $(r_n \circ
  \eta)(T)$ tends to $T$ in operator norm.  The Dynkin functional
  calculus for $T$ is a homomorphism, and so $(r_n \circ \eta)(T) =
  r_n(\eta(T))$.  Since each $r_n(\eta(T))$ is normal, it follows that
  $T$ is also normal.
\end{proof}

As an alternative and a more direct proof of the last lemma, it seems
tempting to argue that if $A = \eta(T)$ is similar to a unitary
operator, then $\eta^{-1}(A)$ is defined, for instance, by the usual
$L^\infty$-functional calculus for normal operators, and
$\eta^{-1}(A)$ is similar to a normal operator.  However, it is not
clear \emph{a priori} why $\eta^{-1}(A) = T$.

Theorem~\ref{two-sided} is a straightforward consequence of Theorem A
and the following lemma.

\begin{lemma}
  \label{lemma-two-sided}
  Let $\Gamma \subset \C$ be a $C^{1+\alpha}$ Jordan curve, and
  $\Omega$ the domain it bounds.  Let $T \in \B(H)$ be an operator
  with $\sigma(T)\subset \Gamma$.  Assume that $\overline{\Omega}$ is
  a $K$-spectral set for $T$ and
  \begin{equation*}
    \|(T-\lambda)^{-1}\|\leq \frac{C}{\dist(\lambda,\Gamma)},\qquad
    \lambda\in\Omega,
  \end{equation*}
  for some constant $C>0$.  Then $T$ is similar to a normal operator.
\end{lemma}

\begin{proof}
  Let $\eta :\overline{\Omega}\to\overline{\D}$ be the Riemann map.
  Since $\partial\Omega$ is of class $C^{1+\alpha}$, then $\eta \in
  \C^{1+\alpha}(\partial\Omega)$ (see, for
  instance,~\cite{Pommerenke}*{Theorem 3.6}).  Extend $\eta$
  pseudoanalytically to $\C\setminus\Omega$.  Now as $\eta$ satisfies
  its assumptions, we can apply Lemma~\ref{lemma-eta}.

  Because $|\eta^n| \leq 1$ in $\overline{\Omega}$ for all $n \geq 0$,
  and $\overline{\Omega}$ is $K$-spectral for $T$, the operator
  $\eta(T)$ is power bounded.  By Theorem~\ref{thm-main}, and the fact
  that $\dist(\lambda,\partial\Omega)$ and $\dist(\eta(\lambda),\T)$
  are comparable,
  \begin{equation*}
    \|(\eta(T)-\lambda)^{-1}\| \leq \frac{C}{1-|\lambda|}, \qquad
    |\lambda|<1.
  \end{equation*}
  Applying Theorem~\ref{vc80} it follows that $\eta(T)$ is similar to
  a unitary operator, and so by Lemma~\ref{similar-unitary}, $T$ is
  similar to a normal operator.
\end{proof}

\begin{proof}[Proof of Theorem~\ref{two-sided}]
  Theorem~A implies that $\overline{\Omega}$ is a complete
  $K$-spectral set for $T$.  It suffices to apply
  Lemma~\ref{lemma-two-sided}.
\end{proof}

\begin{remarks}
  It is straightforward to deduce an analogous result assuming an
  estimate with constant $1$ inside the domain $\Omega$ and an
  estimate with a constant $C$ outside the domain.  Indeed, put $R =
  (T-z_0)^{-1}$, for some fixed $z_0 \in \Omega$.  It follows
  from~\cite{article2}*{Lemma 7} that if $\|(T-\lambda)^{-1}\|\leq
  \dist(\lambda,\Gamma)^{-1}$, then $\|(R-\mu)^{-1}\| \leq
  \dist(\mu,\widetilde{\Gamma})^{-1}$, where $\mu =
  (\lambda-z_0)^{-1}$ and $\widetilde{\Gamma}$ is the image of
  $\Gamma$ under the map $z \mapsto (z-z_0)^{-1}$.  Writing the
  resolvent of $R$ in terms of the resolvent of $T$, it is also easy
  to obtain an estimate for $R$ with a constant $C'> 1$ outside the
  domain bounded by $\widetilde{\Gamma}$.  Since the map $z \mapsto
  (z-z_0)^{-1}$ sends the inside of $\Gamma$ onto the outside of
  $\widetilde{\Gamma}$ and vice versa, it suffices to apply
  Theorem~\ref{two-sided} to $R$.

  The conclusion of Lemma~\ref{lemma-two-sided} is that $T$ is similar
  to a normal operator, and so the set $\overline{\Omega}$ (and even
  the boundary of $\Omega$) must in fact be a complete $K'$-spectral
  set for $T$ for some $K' > 1$.  As it follows from the celebrated
  example of Pisier, combined with the main result of \cite{Petrovic},
  there exists an operator $T$ on $H$ with $\sigma(T)\subset\T$, which
  is polynomially bounded, but not completely bounded.  So, in
  Lemma~\ref{lemma-two-sided} one cannot replace the resolvent
  estimate inside $\Omega$ just by the condition that
  $\sigma(T)\subset \partial\Omega$.  Recently, Gamal'
  \cite{Gamal-Pisier-like} has constructed several new examples of
  operators that are polynomially bounded but not completely
  polynomially bounded.  In particular, any operator given in
  \cite{Gamal-Pisier-like}*{Corollary~2.8} is quasisimilar to an
  absolute continuous unitary operator $U$ and also satisfies
  $\sigma(T)\subset\T$ (the latter follows from
  \cite{Gamal-pbded2015}*{Theorem 2.4}).  Using the techniques
  outlined here, this counterexample can be transferred from
  $\overline\D$ to other sets $\overline{\Omega}$.
\end{remarks}

\section{Mean-square type resolvent estimates}
\label{meansquare}

In this section we give criteria for similarity to a normal operator
analogous to the results by Van~Casteren~\cite{VanCasteren83} and
Naboko~\cite{Naboko} in the context of $C^{1+\alpha}$ Jordan curves.
First of all, a substitute for the curves $r\T$ is needed.  To this
end, we give the following definition.

\begin{definition}
  \label{def:tends_nicely}
  Let $\Gamma \subset \C$ be a Jordan curve and $\Omega$ the region it
  bounds.  A family of Jordan curves $\{\gamma_s\}_{0<s\leq 1}$
  \emph{tends nicely} to $\Gamma$ from the outside as $s\to 0$ if
  $\gamma_s \subset \C \setminus \Omega$ for all $0< s \leq 1$ and the
  following conditions are satisfied for some constant $C \geq 1$:
  \begin{enumerate}[(a)]
  \item $\{\gamma_s\}$ tend uniformly to $\Gamma$,
  \item For all $0 < s \leq 1$, $C^{-1}s \leq \dist(x,\Gamma) \leq
    Cs$, for all $x \in \gamma_s$.
  \item For every $0 < s \leq 1$, $x \in \gamma_s$, and $r > 0$,
    $\operatorname{length}(\gamma_s\cap B(x,r)) \leq Cr$.
  \end{enumerate}
  The family $\{\gamma_s\}_{0 < s \leq 1}$ tends to $\Gamma$ from the
  inside if instead $\gamma_s \subset \Omega$ for all $0 < s \leq 1$
  and again, conditions (a)--(c) are satisfied.
\end{definition}

Condition (c) states that the curves $\gamma_s$ satisfy the
Ahlfors-David condition with a uniform constant.  This condition was
first studied in~\cite{Ahlfors} and~\cite{David}.

If $\Gamma$ is in the class $C^{1+\alpha}$, it is apparent that there
exist a family of curves which tends nicely to $\Gamma$ from the
outside and another family of curves which tends nicely to $\Gamma$
from the inside.  Indeed, let $\eta:U\to\mathbb{C}$ be a function as
in the statement of Lemma~\ref{lemma-eta}, and take
$\Gamma(t)=\eta^{-1}(e^{it})$ for $0\le t\le 2\pi$.  Define
$\gamma^\pm_s$ by $\gamma^\pm_s(t)=\eta^{-1}((1\pm \beta s)e^{it})$,
$0\le t\le 2\pi$.  If the constant $\beta>0$ is small, then
$\gamma^\pm_s \subset U$ for every $0<s\leq 1$.  The curves
$\{\gamma^+_s\}$ tend nicely to $\Gamma$ from outside and the curves
$\{\gamma^-_s\}$ tend nicely to $\Gamma$ from inside.

It will be proved that the mean-square type resolvent estimates
considered here do not depend on the concrete choice of the family of
curves $\{\gamma_s\}$ tending nicely to $\Gamma$.  This will follow
from a lemma concerning Smirnov spaces.

Recall that the Smirnov space $E^2(\Omega,H)$ of $H$-valued function
on a (nice) domain $\Omega$ is defined as the
$L^2(\partial\Omega)$-closure of the $H$-valued rational functions
with poles off $\overline{\Omega}$.  The following lemma dates back to
David and his theorem on the boundedness of certain singular integral
operators on Ahlfors regular curves.  In particular, it follows from
the results in~\cite{David}*{Proposition 6}.

\begin{lemma}
  \label{smirnov}
  Let $\Omega_1, \Omega_2$ be Jordan domains with Ahlfors regular
  boundaries such that $\overline{\Omega}_2\subset \Omega_1$.  If $H$
  is a Hilbert space and $f \in E^2(\Omega_1,H)$, then $f|\Omega_2 \in
  E^2(\Omega_2, H)$ and
  \begin{equation*}
    \|f|\Omega_2\|_{E^2(\Omega_2, H)} \leq C
    \|f\|_{E^2(\Omega_1,H)},
  \end{equation*}
  for some constant $C$ depending only on the Ahlfors constants for
  $\partial \Omega_1$ and $\partial\Omega_2$.
\end{lemma}

\begin{lemma}
  \label{lemma-independent-curves}
  Let $\Gamma$ be a Jordan curve, $T \in \B(H)$ with $\sigma(T)
  \subset \Gamma$, and $\{\gamma_s\}_{0<s\leq 1}$,
  $\{\widetilde{\gamma}_s\}_{0<s\leq 1}$ two families of curves which
  both tend nicely to $\Gamma$ from the inside $($respectively, from
  the outside$)$.  If
  \begin{equation*}
    \int_{\gamma_s} \|(T-\lambda)^{-1}x\|^2\,|d\lambda| \leq
    \frac{C\|x\|^2}{s}, \qquad x \in H, 0<s\leq 1,
  \end{equation*}
  for some constant $C$ independent of $x$ and $s$, then
  \begin{equation*}
    \int_{\widetilde{\gamma}_s} \|(T-\lambda)^{-1}x\|^2\,|d\lambda|
    \leq \frac{C'\|x\|^2}{s}, \qquad x \in H, 0<s\leq 1,
  \end{equation*}
  for some constant $C'$ independent of $x$ and $s$.
\end{lemma}

\begin{proof}
  First assume that $\{\gamma_s\}$ and $\{\widetilde{\gamma}_s\}$ tend
  nicely to $\Gamma$ from the inside.  Denote by $\Omega_s$ the domain
  bounded by $\gamma_s$ and by $\widetilde{\Omega}_s$ the domain
  bounded by $\widetilde{\gamma}_s$.

  By definition, since $\{\widetilde{\gamma}_s\}$ tends nicely to
  $\Gamma$, there exists $0<s_0 \leq 1$ such that for all $s \in
  (0,s_0)$, the closure of $\widetilde{\Omega}_s$ is contained in
  $\Omega_1$.  Then by continuity of $1/s$ on $[s_0,1]$, there exists
  a constant $\tilde C$ such that the claim holds whenever $s\in
  [s_0,1]$.

  So assume that $s \in (0,s_0)$.  Applying Lemma~\ref{smirnov} to the
  function $f(z) = (T-z)^{-1}x$ and the domains $\Omega_1 = \Omega_s$,
  $\Omega_2 = \widetilde{\Omega}_t$, to obtain
  \begin{equation*}
    \begin{split}
      & \int_{\widetilde{\gamma}_s} \|(T-\lambda)^{-1}x\|^2\,|d\lambda|
      = \|f|\widetilde{\Omega}\|_{E^2(\widetilde{\Omega}_s,H)} \leq
      K\|f\|_{E^2(\Omega_1,H)} \\ =&\, K\int_{\gamma_1}
      \|(T-\lambda)^{-1}x\|^2\,|d\lambda| \leq KC\|x\|^2 \leq
      \frac{KC\|x\|^2}{s_0} \leq \frac{KC\|x\|^2}{s}.
    \end{split}
  \end{equation*}

  If $\{\gamma_s\}$ and $\{\widetilde{\gamma}_s\}$ tend nicely to
  $\Gamma$ from the outside, choose a point $z_0$ inside the domain
  bounded by $\Gamma$ and apply an inversion: $z \mapsto
  (z-z_0)^{-1}$.  Since
  \begin{equation*}
    ((T-z_0)^{-1}-(\lambda-z_0)^{-1})^{-1} =
    (\lambda-z_0)(T-z_0)(T-\lambda)^{-1},
  \end{equation*}
  the bounds for the resolvent of $T$ imply equivalent bounds for the
  resolvent of $(T-z_0)^{-1}$, and conversely.  Thus, this case
  follows from the previous one.
\end{proof}

It is well known that, in the context of the unit circle, a resolvent
bound of mean-square type implies a pointwise resolvent bound such as
\eqref{eq:resolvent-estimate}.  The proof of this fact uses the usual
pointwise estimate for an $H^2$ function in the disk, which involves
the norm of the reproducing kernel.  The following lemma is a
generalization of this.

\begin{lemma}
  \label{lmean->lrg}
  Let $\Gamma \subset \C$ be a Jordan curve of class $C^{1+\alpha}$,
  $\Omega$ the region it bounds and $T \in \B(H)$ with
  $\sigma(T)\subset \Gamma$.  If
  \begin{equation}
    \label{eq:inside}
    \int_{\gamma_s} \|(T-\lambda)^{-1}x\|^2\,|d\lambda| \leq
    \frac{C\|x\|^2}{s}, \qquad x\in H,\ 0<s\leq 1,
  \end{equation}
  for some constant $C$ independent of $x$ and $s$ and some family of
  curves $\{\gamma_s\}$ which tends nicely to $\Gamma$ from the inside
  $($respectively, outside$)$, then
  \begin{equation*}
    \|(T-\lambda)^{-1}\| \leq \frac{C'}{\dist(\lambda,\Gamma)}, \qquad
    \lambda \in \Omega\ (\text{respectively,} \;\lambda \in
    \C\setminus\overline{\Omega}),
  \end{equation*}
  for some constant $C'$ independent of $\lambda$.
\end{lemma}

\begin{proof}
  Assume that $\{\gamma_s\}$ tends nicely to $\Gamma$ from the inside.
  Let $\eta$ be a function as in the statement of
  Lemma~\ref{lemma-eta}, and $U$ the neighborhood of $\Gamma$ that
  appears in that lemma.  Fix $\lambda \in \Omega\cap U$.  Then
  because $\eta$ is bi-Lipschitz, $t = \dist(\lambda,\Gamma)$ is
  comparable to $\dist(\eta(\lambda),\T)$.  Put $r = 1-
  \dist(\eta(\lambda),\T)/2$.

  Now consider the Jordan curve $\Lambda=\eta^{-1}(r\T)$.  This lies
  inside $\Omega$ and $\dist(z,\Gamma)$ is comparable to $t$ for every
  $z \in \Lambda$.  Therefore, it is possible to choose $0<s\leq 1$
  such that $t \leq C_1 s$, and $\Lambda$ is inside the region bounded
  by $\gamma_s$.

  Fix $x \in H$ and put $f(z) = (\eta(T)-z)^{-1}x$, and $g(z) =
  f(z/r)$.  By the usual pointwise estimate for a function in $H^2$,
  the Hardy space of the disk,
  \begin{equation*}
    \begin{split}
      \|g(z)\| & \leq (1-|z|^2)^{-1/2}\|g\|_{E^2(\D,H)} =
    (1-|z|^2)^{-1/2}\|f\|_{E^2(r\D,H)} \\
    & \leq (1-|z|^2)^{-1/2} \|f\|_{E^2(W,H)},
    \end{split}
  \end{equation*}
  where $W$ is the domain bounded by $\eta(\gamma_s)$ and the last
  inequality comes from Lemma~\ref{smirnov}.  Now by
  Theorem~\ref{thm-main} and \eqref{eq:inside},
  \begin{equation*}
    \begin{split}
      & \|f\|^2_{E^2(W,H)} = \int_{\eta(\gamma_s)}
      \|(\eta(T)-z)^{-1}x\|^2\,|dz| \\ \leq &\, C_1
      \int_{\gamma_s}\|(\eta(T)-\eta(w))^{-1}x\|^2\,|dw| \leq C_2
      \int_{\gamma_s}\|(T-w)^{-1}x\|^2\,|dw| \leq
      \frac{C_3\|x\|^2}{s}.
    \end{split}
  \end{equation*}
  Hence,
  \begin{equation*}
    \|(\eta(T)-z/r)^{-1}x\|^2 \leq C_3(1-|z|^2)^{-1}s^{-1}\|x\|^2,
  \end{equation*}
  and the inequality above is valid for all $x \in H$.  Putting $z =
  r\eta(\lambda)$ yields
  \begin{equation*}
    \|(\eta(T)-\eta(\lambda))^{-1}\|^2 \leq
    C_3(1-|r\eta(\lambda)|^2)^{-1}s^{-1}
    \leq C_4 t^{-2}.
  \end{equation*}
  By another application of Theorem~\ref{thm-main},
  \begin{equation*}
    \|(T-\lambda)^{-1}\|\leq \frac{C'}{t} =
    \frac{C'}{\dist(\lambda,\Gamma)}.
  \end{equation*}

  The case when $\{\gamma_s\}$ tends nicely to $\Gamma$ from the
  outside is proved by applying the inversion $z \mapsto
  (z-z_0)^{-1}$, as in the proof of
  Lemma~\ref{lemma-independent-curves}.
\end{proof}

We now state and prove generalizations of Theorems \ref{vc83} and
\ref{n} in Theorems~\ref{van-casteren} and Theorem~\ref{naboko}.  The
proofs both follow the same line of reasoning, using the tools so far
developed to pass to $\T$ and then applying Van~Casteren's or Naboko's
theorem.  It is worth highlighting that, as with the original
theorems, there is no easy way to deduce either result from the other.

\begin{theorem}[Van~Casteren-type theorem for curves]
  \label{van-casteren}
  Let $\Gamma \subset \C$ be a Jordan curve of class $C^{1+\alpha}$,
  $\Omega$ the region it bounds and $T\in\B(H)$ with
  $\sigma(T)\subset\Gamma$.  Let $\{\gamma_s\}_{0<s\leq 1}$ be a
  family of curves which tends nicely to $\Gamma$ from the outside.
  Then $T$ is similar to a normal operator if and only if the
  following three conditions are satisfied.
  \begin{equation*}
    \|(T-\lambda)^{-1}\|\leq \frac{C}{\dist(\lambda,\Gamma)}, \qquad
    \lambda\in\Omega,
  \end{equation*}
  \begin{equation}
    \label{eq:int1}
    \int_{\gamma_s} \|(T-\lambda)^{-1}x\|^2\,|d\lambda| \leq
    \frac{C\|x\|^2}{s}, \qquad x \in H,\ 0<s\leq 1,
  \end{equation}
  \begin{equation}
    \label{eq:int2}
    \int_{\gamma_s} \|(T^*-\overline{\lambda})^{-1}x\|^2\,|d\lambda|
    \leq \frac{C\|x\|^2}{s}, \qquad x \in H,\ 0<s\leq 1.
  \end{equation}
\end{theorem}

\begin{proof}
  First, assume that $T$ satisfies the three resolvent conditions.
  Let $\eta:U\to \mathbb{C}$ be a function as in the statement of
  Lemma~\ref{lemma-eta}.  Take $\gamma_s \subset U$ for every $0<s\leq
  1$.  By Lemma~\ref{lmean->lrg}, the operator $T$ satisfies the
  resolvent estimate \eqref{eq:resolvent-estimate}, so $\eta(T)$ is
  defined by the $C^{1+\alpha}(\Gamma)$-functional calculus for $T$.
  By Theorem~\ref{thm-main} and the fact that $\dist(\lambda,\Gamma)$
  and $\dist(\eta(\lambda),\T)$ are comparable,
  \begin{equation*}
    \|(\eta(T)-\lambda)^{-1}\| \leq \frac{C_1}{1-|\lambda|},\qquad
    |\lambda| < 1,
  \end{equation*}
  as well as
  \begin{equation*}
    \int_{\eta(\gamma_s)} \|(\eta(T)-\lambda)^{-1}x\|^2\,|d\lambda|
    \leq \frac{C_2\|x\|^2}{s},\qquad x\in H,\ 0 < s \leq 1,
  \end{equation*}
  which follows by making a change of variables $\lambda = \eta(\mu)$
  and applying Theorem~\ref{thm-main}.

  Since $\eta:U\to \eta(U)$ is a $C^1$ diffeomorphism and
  bi-Lipschitz, the family of curves $\{\eta(\gamma_s)\}_{0<s\leq 1}$
  tends nicely to $\T$ from the outside.  Therefore, by
  Lemma~\ref{lemma-independent-curves} applied to the family
  $\widetilde{\gamma_s}=(1+s)\T$,
  \begin{equation*}
    \int_{|\lambda|=r} \|(\eta(T)-\lambda)^{-1}x\|^2\,|d\lambda| \leq
    \frac{C_3\|x\|^2}{r-1},\qquad x\in H,\ 1 < r < 2.
  \end{equation*}

  Similar reasoning with $T^*$ in place of $T$ and
  $\widetilde{\eta}(z) = \overline{\eta(\overline{\zeta})}$ in place
  of $\eta$ shows that
  \begin{equation*}
    \int_{|\lambda|=r} \|(\eta(T)^*-\lambda)^{-1}x\|^2\,|d\lambda|
    \leq \frac{C_4\|x\|^2}{r-1},\qquad x\in H,\ 1 < r < 2,
  \end{equation*}
  as $\widetilde{\eta}(T^*) = \eta(T)^*$.

  Now apply Theorem~\ref{vc83} to deduce that $\eta(T)$ is similar to
  a unitary operator.  Then by Lemma~\ref{similar-unitary}, $T$ is
  similar to a normal operator.

  Conversely, assume that $T$ is similar to a normal operator and
  $\sigma(T)\subset \Gamma$.  Replacing $T$ by $STS^{-1}$ if
  necessary, it can be assumed that $T$ is normal.  The first
  condition on the resolvent of $T$ holds, since $\|(T-\lambda)^{-1}\|
  \leq \dist(\lambda,\Gamma)^{-1}$ for a normal operator $T$.

  The operator $\eta(T)$ is unitary, so by expressing the resolvent as
  a power series and using the fact that $\|\eta(T)^n\| = 1$ for all
  $n\geq 0$, it follows that
  \begin{equation*}
    \int_{r\T} \|(\eta(T)-\lambda)^{-1}x\|^2\, |d\lambda| \leq
    \frac{C_5\|x\|^2}{r-1},\qquad 1<r<2,\ x\in H.
  \end{equation*}
  By Theorem~\ref{thm-main} (although since $T$ is normal, a simpler
  argument could be devised), for some constant $C>0$ and some
  neighborhood $U$ of $\Gamma$,
  \begin{equation*}
    C^{-1}\|(T-\lambda)^{-1}x\|\leq \|(\eta(T)-\eta(\lambda))^{-1}x\|
    \leq C\|(T-\lambda)^{-1}x\|,\quad \lambda \in U\setminus \Gamma,\
    x\in H,
  \end{equation*}

  Therefore,
  \begin{equation*}
    \int_{\eta^{-1}(r\T)} \|(T-\lambda)^{-1}x\|^2\, |d\lambda| \leq
    \frac{C_6\|x\|^2}{r-1},\qquad 1<r<2,\ x\in H.
  \end{equation*}
  The family of curves $\widetilde{\gamma_s} = \eta^{-1}((1+s)\T)$
  tends nicely to $\Gamma$.  Apply
  Lemma~\ref{lemma-independent-curves} to get that $T$ satisfies
  \eqref{eq:int1}.  Use of similar reasoning, but with $T^*$ instead
  of $T$ and $\widetilde{\eta}$ instead of $\eta$ yields the
  inequality \eqref{eq:int2}.
\end{proof}

Sz.-Nagy proved in~\cite{SzNagy} that an operator $T$ is similar to a
unitary operator if and only if $\|T^n\|\leq C$ for all $n\in \Z$.  It
is easy to use this result to show that if $\sigma(T)\subset\Gamma$,
the resolvent estimate \eqref{eq:resolvent-estimate} holds and
$\|\eta^n(T)\|\leq C$ for all $n\in\Z$, then $T$ is similar to a
normal operator.

The following corollary is a generalization of Theorem~\ref{vc80}.
Note that here it is only assumed that $\|\eta(T)^n\|\leq C$ for all
$n \geq 0$.  The proof is similar to the proof of
Theorem~\ref{van-casteren}, but Theorem~\ref{vc80} is used instead of
Theorem~\ref{vc83}.

\begin{corollary}
  Let $\Gamma \subset \C$ be a Jordan curve of class $C^{1+\alpha}$,
  and $T\in\B(H)$ with $\sigma(T)\subset\Gamma$.  Assume that
  \begin{equation*}
    \|(T-\lambda)^{-1}\|\leq \frac{C}{\dist(\lambda,\Gamma)}, \qquad
    \lambda\in\C\setminus\Gamma.
  \end{equation*}
  Let $\eta:\Gamma\to\T$ be a function as in the statement of
  Lemma~\ref{lemma-eta}.  Define the operator $\eta(T)$ by the
  $C^{1+\alpha}$-functional calculus.  If $\eta(T)$ is power bounded,
  then $T$ is similar to a normal operator.
\end{corollary}

\begin{theorem}[Naboko-type theorem for curves]
  \label{naboko}
  Let $\Gamma \subset \C$ be a Jordan curve of class $C^{1+\alpha}$,
  $\Omega$ the region it bounds and $T\in\B(H)$ with
  $\sigma(T)\subset\Gamma$.  Let $\{\gamma_s\}_{0<s\leq 1}$ be a
  family of curves which tends nicely to $\Gamma$ from the outside and
  $\{\widetilde{\gamma}_s\}$ a family of curves which tends nicely to
  $\Gamma$ from the inside.  Then $T$ is similar to a normal operator
  if and only if the following two conditions are satisfied for all $x
  \in H$ and $0<s\leq 1$:
  \begin{equation*}
    \int_{\gamma_s} \|(T-\lambda)^{-1}x\|^2\,|d\lambda| \leq
    \frac{C\|x\|^2}{s} \quad \text{and} \quad
    \int_{\widetilde{\gamma}_s}
    \|(T^*-\overline{\lambda})^{-1}x\|^2\,|d\lambda| \leq
    \frac{C\|x\|^2}{s}.
  \end{equation*}
\end{theorem}

\begin{proof}
  The proof of this theorem is like the proof of
  Theorem~\ref{van-casteren}.  If $T$ satisfies the two conditions in
  the statement of this theorem, instead of using Van~Casteren's
  theorem, use Naboko's Theorem~\ref{n} to show that $\eta(T)$ is
  similar to a unitary operator.  We only sketch the proof.

  First, Lemma~\ref{lmean->lrg} implies that $T$ satisfies the
  resolvent estimate \eqref{eq:resolvent-estimate}.  Choose a function
  $\eta$ as in Lemma~\ref{lemma-eta}.  The operator $\eta(T)$ is well
  defined.  By Theorem~\ref{thm-main} and
  Lemma~\ref{lemma-independent-curves},
  \begin{equation*}
    \int_{|\lambda|=r} \|(\eta(T)-\lambda)^{-1}x\|^2\,|d\lambda| \leq
    \frac{C_1\|x\|^2}{r-1},\qquad x\in H,\ 1 < r < 2
  \end{equation*}
  and
  \begin{equation*}
    \int_{|\lambda|=r} \|(\eta(T)^*-\lambda)^{-1}x\|^2\,|d\lambda| \leq
    \frac{C_2\|x\|^2}{1-r},\qquad x\in H,\ 0 < r < 1.
  \end{equation*}
  Theorem~\ref{n}, then gives that $\eta(T)$ is similar to a unitary
  operator.  It follows that $T$ is similar to a normal operator by
  Lemma~\ref{similar-unitary}.

  The converse direction is proved as in Theorem~\ref{van-casteren}.
\end{proof}

\section{Some comments and examples}
\label{sec-examples}

So far we have only discussed operators with spectrum on some smooth
curve in $\C$.  Similar results can be presented if the spectrum is
allowed to be a union of a smooth curve and a sequence of points
tending to this curve.  In~\cite{BenamaraNikolski}, Benamara and
Nikolski show that a contraction $T$ with finite defects is similar to
a normal operator if and only if $\sigma(T)\neq\overline{\D}$ and
$\|(T-\lambda)^{-1}\| \leq C\dist(\lambda,\sigma(T))^{-1}$ for all
$\lambda\in\C\setminus\sigma(T)$.  For such a contraction,
$\sigma(T)\setminus\T$ is always a Blaschke sequence in $\D$.
Moreover, Benamara and Nikolski prove that the resolvent estimate
forces $\sigma(T) \cap \D$ to be quite sparse (more precisely, it has
to satisfy the $\Delta$-Carleson condition).  Later in~\cite{Kupin},
Kupin studied contractions with infinite defects.  He proved that if
the spectrum of a contraction $T$ is not all $\overline{\D}$, and if
it satisfies both $\|(T-\lambda)^{-1}\| \leq
C\dist(\lambda,\sigma(T))^{-1}$ for all $\lambda \in \sigma(T)$ and
the so-called Uniform Trace Boundedness condition, then it is similar
to a normal operator.  An analogue of Uniform Trace Boundedness
condition for dissipative operators was given by Vasyunin and Kupin
in~\cite{KupinVasyunin}, and then applied to integral operators.
In~\cite{Kupin2003}, Kupin also uses the Uniform Trace Boundedness
condition to give conditions for an operator similar to a contraction
to be similar to a normal operator.

On the other hand, Kupin and Treil showed in~\cite{KupinTreil} that if
$T$ is a contraction with $\sigma(T) \neq \overline{\D}$ and
$\|(T-\lambda)^{-1}\| \leq C\dist(\lambda, \sigma(T))^{-1}$ but one
only assumes that $I-T^*T$ is trace class (instead of finite rank),
then $T$ need not be similar to a normal operator, thus solving a
conjecture in~\cite{BenamaraNikolski}.

All these results concern operators with \emph{thin} spectrum (in
other words, those for which the area of the spectrum is zero).  For
operators having \emph{thick} spectrum (so non-zero area), in general
there is no hope of obtaining criteria for similarity to a normal
operator solely in terms of resolvent operator estimates.  Indeed, for
any hyponormal operator $T$, the best possible estimate
$\|(T-\lambda)^{-1}\| =\dist(\lambda, \sigma(T))^{-1}$ holds, and for
any compact set $F$ of positive area there exists a hyponormal
operator $T$ not similar to a normal one with $\sigma(T)=F$.

Resolvent conditions for similarity to other classes of operators,
such as selfadjoint operators or isometries, have also been considered
in the literature.  Faddeev gives conditions in~\cite{Faddeev} for
similarity to an isometry in the case where $\dim \ker (T^* - \lambda
I) = 1$ for all $\lambda \in \D$.  In~\cite{Popescu}, Popescu also
states several conditions for similarity to an isometry.

In~\cite{Malamud}, Malamud gives a series of abstract conditions for
an operator $A$ to be similar to a selfadjoint one.  His conditions
involve, in particular, resolvent estimates of the form
$\|V^{1/2}(A-\lambda)^{-1}\| \leq C|\Im \lambda|^{-1/2}$, where $V =
|\Im A|$.  These estimates are related to Theorem~\ref{n}.  He applies
his results to a triangular operator on $L^2([0,1],d\mu)$ of the form
\begin{equation}
  \label{triang}
  (Af)(x) = \alpha(x) f(x) + i\int_x^1 K(x,t)f(t)\,d\mu(t),
\end{equation}
for an Hermitian kernel $K(x,t)$.

Naboko and Tretter used Theorem~\ref{n} in~\cite{NabokoTretter} to
examine operators of the form~\eqref{triang}, where
$K(x,t)=\phi(x)\psi(t)$ (so that $K(x,t)$ need not be Hermitian) and
$\phi\psi \equiv 0$.  By using our criteria for curves, most likely
Naboko and Tretter's results can be extended to analogous operators on
$L^2$ of a curve, though it might be rather technical.

Also concretely, such conditions have been used in the study of
differential operators.  For example, in~\cite{FaddeevShterenberg2},
Faddeev and Shterenberg use a version of Theorem~\ref{n} in examining
similarity to a selfadjoint operator for operators of the form $A =
-\frac{\operatorname{sign}x}{|x|^\alpha p(x)} \frac{d^2}{dx^2}$, where
$\alpha > -1$ and $p$ is a positive function which is bounded above
and below.  Their criteria were further generalized by Karabash,
Kostenko and Malamud in~\cite{KarabKostMalamud}.

Resolvent growth conditions for similarity to a unitary operator have
also been used in the study of Toeplitz operators with unimodular
symbol; see~\cites{Clark, Peller, Gamal-Toe} and references therein.

Article~\cite{ElFallahRansf2004} contains a discussion of the
relationship between the growth of powers of an operator $T$ with
$\sigma(T)\subset \overline{\D}$, first order growth of its resolvent
outside $\overline{\D}$ and the size of the set $\sigma(T)\cap \T$.

The conditions for a contraction $T$ to be similar to a unitary
operator in terms of the characteristic function of $T$ are well
known.  Given a contraction $T\in\B(H)$, one defines defect operators
$D_T = (I-T^*T)^{\frac{1}{2}}$, $D_{T^*} = (I-TT^*)^{\frac{1}{2}}$ and
defect spaces $\mathfrak{D}_T = \overline{D_T H}$, $\mathfrak{D}_{T^*}
= \overline{D_T^* H}$.  For $\lambda \in \D$, the \emph{characteristic
  function} $\Theta_T(\lambda) : \mathfrak{D}_T \to
\mathfrak{D}_{T^*}$ is given by
\begin{equation*}
  \Theta_T(\lambda) = [-T + \lambda D_{T^*}(I-\lambda
  T^*)^{-1}D_T]|\mathfrak{D}_T.
\end{equation*}
As Sz.-Nagy and Foias proved in~\cite{SzNagyFoias65} (see also
Sz.-Nagy and Foias~\cite{SzNagyFoiasBook}*{Chapter~9}), $T$ is similar
to a unitary operator if and only if $\Theta_T(\lambda)$ is invertible
for all $\lambda \in \D$ and
\begin{equation}
  \label{eq:NF-criterion}
  \sup_{\lambda\in\D} \|\Theta_T(\lambda)^{-1}\| < \infty.
\end{equation}
L.A.~Saknovich extended the results of Sz.-Nagy and Foias to operators
which are not necessarily contractions in~\cite{Sakhnovich}.
Saknovich's condition is only sufficient for similarity to a unitary
operator and not necessary in general.  See also
Naboko~\cite{Naboko-1981sing}*{Theorem~12}.

In a series of articles, Naboko constructed and studied a functional
model for non-dissipative perturbations of self-adjoint operators.  A
detailed exposition of this model can be found in~\cite{Naboko-1981}.
In that paper, the problem of existence of wave operators in this
context is discussed.  A functional model for perturbations of normal
operators with spectrum on a curve, extending Naboko's model, has been
developed by Tikhonov in~\cite{Tikhonov-A} and subsequent papers.

\begin{example}
  A purely contractive function $\Theta$ can be chosen satisfying the
  condition \eqref{eq:NF-criterion} and the Sz.-Nagy Foias model used
  to construct a completely non-unitary contraction $T$ such that
  $\Theta_T = \Theta$.  Such a contraction is non-unitary and similar
  to a unitary operator, and so $\sigma(T) \subset \mathbb T$.  Hence
  \begin{equation*}
    \|(T-\lambda)^{-1}\| \leq \frac{C}{1-|\lambda|},\qquad
    |\lambda|<1.
  \end{equation*}
  Since $T$ is also a contraction, by von~Neumann's inequality
  \begin{equation*}
    \|(T-\lambda)^{-1}\| \leq \frac{1}{|\lambda|-1},\qquad
    |\lambda|>1.
  \end{equation*}
  Recall that, by Stampfli's theorem stated in the introduction, if
  under these conditions $T$ satisfies
  \begin{equation*}
    \|(T-\lambda)^{-1}\|\leq \frac{1}{||\lambda|-1|}, \qquad
    |\lambda|\neq 1,
  \end{equation*}
  then $T$ must be normal.  Thus this yields an example of an operator
  which satisfies the hypotheses of Theorem~\ref{two-sided} and is
  non-normal.
\end{example}

There are also examples of this type among the class of
$\rho$-contractions.  If $\rho > 0$, an operator $T \in \B(H)$ is
called a \emph{$\rho$-contraction} if there is a larger Hilbert space
$K \supset H$ and a unitary $U \in \B(K)$ such that
\begin{equation*}
  T^n = \rho P_HU^n|H,\qquad n=1,2,\ldots,
\end{equation*}
where $P_H$ denotes the orthogonal projection onto $H$.  The classes
$\mathcal{C}_\rho$ of $\rho$-contractions are nested and increasing
with $\rho$, and the class $\mathcal{C}_1$ coincides with the class of
contractions.

If $T$ is a $\rho$-contraction with $\rho \geq 2$, then
\begin{equation*}
  \|(T-\lambda)^{-1}\| \leq \frac{1}{|\lambda|-1},\qquad 1 < |\lambda|
  < \frac{\rho-1}{\rho-2}.
\end{equation*}
(Here $\frac{\rho-1}{\rho-2} = +\infty$ if $\rho = 2$.)  Therefore,
any $\rho$-contraction which is similar to a unitary operator also
satisfies the hypotheses of Theorem~\ref{two-sided}.  If $T$ is a
$2$-contraction, then one may take the set $U = \C$ in the hypotheses
of Theorem~\ref{two-sided}.  However, for a $\rho$-contraction with
$\rho > 2$, $U$ will in general be a smaller set.

It is natural to ask if there is an example of a $\rho$-contraction
which is not a contraction and where the spectrum is contained in the
unit circle.  Stampfli shows that this can occur in
\cite{Stampfli3}*{Example~2}.  There, $\rho=2$ and the spectrum of the
operator is a single point.  He proves that since the spectrum is
countable, this operator must be normal.

Another example, this time with a bilateral weighed shift, is given
below.  Here the spectrum is the whole unit circle, and the operator
is similar to a unitary, but is not normal.

Any such example must have uncountable spectrum.  Consequently, it
would be interesting to know for a non-normal operator $T$ with
spectrum $\sigma(T)$ contained in a curve $\Gamma$ and satisfying the
hypotheses of Theorem~\ref{two-sided} (so that it is similar to a
normal operator), just how small $\sigma(T)$ can be.  See
\cite{ElFallahRansf2004} and references therein for a discussion of
some similar questions.

\begin{example}
  \label{example-weighted-shift}
  Assume that $\alpha,\beta > 0$, $\max(\alpha,\beta) > 1$, and
  $\alpha^2 + \beta^2 \leq 4$.  Let $T$ be the bilateral weighted
  shift $T$ on $\ell^2(\Z)$ with weights
  $\{\ldots,1,1,\boxed{\alpha},\beta,1,1,\ldots\},$ defined by
  \begin{equation*}
    T(\ldots,x_{-1},\boxed{x_0},x_1, x_2,\ldots) =
    (\ldots,x_{-2},\boxed{x_{-1}},\alpha x_0, \beta x_1, x_2, \ldots).
  \end{equation*}
  (Here $\boxed{\phantom{\cdot}}$ indicates the $0$-th component).
  Then $T$ is a $2$-contraction which is not a contraction, yet is
  similar to a unitary operator.
\end{example}

Obviously, $\|T\| = \max(\alpha,\beta,1) > 1$.  Since $\alpha,\beta >
0$, the operator $T$ is similar to the unitary bilateral shift $U$ on
$\ell^2(\Z)$ with all weights equal to $1$.  It remains to show that
$T$ is a $2$-contraction.

Recall that $T$ is a $2$-contraction if and only if
\begin{equation*}
  \operatorname{Re}(\theta T) \leq I, \qquad |\theta| = 1.
\end{equation*}
Since $\theta T$ is unitarily equivalent to $T$ when $|\theta|=1$, it
is enough to check this inequality for $\theta = 1$.

Put $A = 2 \Re T$, and check that $\sigma(A) \cap (2,+\infty) =
\emptyset$.  Since $A$ is a finite rank perturbation of $U+U^*$ and
$\sigma(U+U^*) = [-2,2]$, it suffices to show that $A$ has no
eigenvalues in $(2,+\infty)$.  Assume that $x = (x_n)_{n\in\Z}$ is a
non-zero vector in $\ell^2(\Z)$ that satisfies $(A-\lambda)x = 0$ for
some $\lambda > 2$.  This means that
\begin{equation}
  \label{eq:*}
  x_n - \lambda x_{n+1} + x_{n+2} = 0, \quad |n| \geq 2,
\end{equation}
\begin{equation}
  \label{eq:**}
  x_{-1} - \lambda x_0 + \alpha x_1 = 0,
\end{equation}
\begin{equation}
  \label{eq:***}
  \alpha x_0 - \lambda x_1 + \beta x_2 = 0,
\end{equation}
\begin{equation}
  \label{eq:****}
  \beta x_1 - \lambda x_2 + x_3 = 0.
\end{equation}

Put
\begin{equation*}
  u_\pm = u_\pm(\lambda) = \frac{\lambda \pm \sqrt{\lambda^2-4}}{2}\, .
\end{equation*}
Then \eqref{eq:*} and $x\in \ell^2(\Z)$ imply that for some non-zero
$a,b$,
\begin{equation*}
  \begin{split}
    x_n &= au_-^n,\qquad n \geq 2,\\
    x_n &= bu_+^n,\qquad n \leq 0.
  \end{split}
\end{equation*}
The quotients $y_n = \frac{x_{n+1}}{x_n}$ satisfy
\begin{equation*}
  \begin{split}
    y_n &= u_-,\qquad n \geq 2,\\
    y_n &= u_+, \qquad n \leq -1.
  \end{split}
\end{equation*}

If
\begin{equation*}
  F x_{n} - \lambda x_{n+1} + G x_{n+2} = 0,
\end{equation*}
then $y_n$ is obtained from $y_{n+1}$ by applying the M\"obius
transformation $z \mapsto \frac{F}{-Gz + \lambda}$, which can be
encoded by the $2\times 2$ matrix { \tiny $\begin{pmatrix} 0 & F \\ -G
    & \lambda \end{pmatrix}$}.  The composition of M\"obius
transformations reduces to multiplying the corresponding $2\times 2$
matrices, so equations \eqref{eq:**}--\eqref{eq:****} yield
\begin{equation*}
  u_+(\lambda) = y_{-1} = \frac{M_{11}(\lambda)y_2 +
    M_{12}(\lambda)}{M_{21}(\lambda)y_2 + M_{22}(\lambda)} =
  \frac{M_{11}(\lambda)u_-(\lambda) +
    M_{12}(\lambda)}{M_{21}(\lambda)u_-(\lambda) + M_{22}(\lambda)},
\end{equation*}
where
\begin{equation*}
  \begin{pmatrix}
    M_{11}(\lambda) & M_{12}(\lambda)\\ M_{21}(\lambda) &
    M_{22}(\lambda)
  \end{pmatrix}
  =
  \begin{pmatrix}
    0 & 1\\ -\alpha & \lambda
  \end{pmatrix}
  \begin{pmatrix}
    0 & \alpha\\ -\beta & \lambda
  \end{pmatrix}
  \begin{pmatrix}
    0 & \beta\\ -1 & \lambda
  \end{pmatrix}.
\end{equation*}

Putting
\begin{equation*}
  \begin{split}
    f(\lambda) &= u_+(\lambda)\big(M_{21}(\lambda)u_-(\lambda) +
    M_{22}(\lambda)\big) - \big(M_{11}(\lambda)u_-(\lambda) +
    M_{12}(\lambda)\big) =\\
    &= \lambda[(\lambda^2 - (\alpha^2+\beta^2))u_+(\lambda) +
    u_-(\lambda)],
  \end{split}
\end{equation*}
it follows that $f(\lambda) = 0$.  However, since $\lambda > 2$,
$\alpha^2 + \beta^2 \leq 4$, and $u_+(\lambda)$ and $u_-(\lambda)$ are
positive, $f(\lambda) > 0$.  This is a contradiction.

\medskip

In the above argument, the numerical radius of the weighted shift $T$
was computed by examining the spectrum of its real part.  There are
several works in the literature devoted to the study of the numerical
radius of weighted shifts, using similar techniques.
See~\cites{UndrakhEtl} and references therein.

In~\cite{AndoTakahashi}, And\^o and Takahashi proved that if an
operator $T$ is polynomially bounded and there exist an injective
operator $X$ and a unitary operator $W$ with non-singular spectral
measure with respect to the Lebesgue measure on $\T$, and such that
$XT = WX$, then $T$ is similar to a unitary operator.  Moreover, if
such $T$ is also a $\rho$-contraction for some $\rho > 0$, then $T$ is
itself unitary.  This does not apply in
Example~\ref{example-weighted-shift}, since the operator $T$ is
similar to the bilateral shift in $L^2(\T)$, the spectral measure of
which is not singular.  A similar result is contained in
Mlak~\cite{Mlak}.  See Gamal'~\cite{Gamal} and the references therein
for extensions of these results.

\begin{example}
  One can easily construct non-normal operators which satisfy the
  hypotheses of Theorem~\ref{two-sided} for a Jordan domain $\Omega
  \neq \D$.  Let $A$ be a non-unitary contraction which is similar to
  a unitary operator.  Take a Riemann mapping $\varphi:\Omega \to \D$
  and put $\psi = \varphi^{-1}$.  The operator $T = \varphi(A)$ is
  well defined and non-normal.  If $\lambda \in
  \C\setminus\overline{\Omega}$, then by von~Neumann's inequality
  \begin{equation*}
    \|(T-\lambda)^{-1}\| \leq
    \|(\varphi-\lambda)^{-1}\|_{H^\infty(\D)} =
    \dist(\lambda,\Omega)^{-1}.
  \end{equation*}
  If $\lambda \in \Omega$, the inequality
  \begin{equation*}
    \|(T-\lambda)^{-1}\|\leq C\dist(\lambda,\Omega)^{-1}
  \end{equation*}
  follows from the fact that $T$ is similar to a normal operator.  The
  operator $T$ satisfies the hypotheses of Theorem~\ref{two-sided}.
\end{example}

It is not obvious how to use a Riemann mapping in a similar manner to
get a result analogous to Example~\ref{example-weighted-shift} for a
general Jordan domain $\Omega$.

\end{document}